\documentclass[12pt]{amsart}

\usepackage{amsthm}
\usepackage{amsmath,amssymb}
\usepackage{color}
\usepackage[all]{xy}

\begin{document}

\newtheorem{theorem}{Theorem}
\newtheorem{proposition}[theorem]{Proposition}
\newtheorem{corollary}[theorem]{Corollary}
\newtheorem{lemma}[theorem]{Lemma}
\newtheorem{definition}[theorem]{Definition}
\newtheorem{remark}[theorem]{Remark}

\newcommand\C{{\mathbb C}}
\newcommand\Z{{\mathbb Z}}
\newcommand\N{{\mathbb N}}
\newcommand\K{{\mathbb K}}
\newcommand\R{{\mathcal R}}
\newcommand\A{{\mathcal A}}
\newcommand\AV{{\mathcal {AV}}}
\newcommand\F{{\mathcal F}}
\newcommand\V{{\mathcal V}}
\newcommand\LL{{\mathcal L}}
\newcommand\HH{{\mathcal H}}
\newcommand\h{{\mathfrak h}}
\newcommand\n{{\mathfrak n}}
\newcommand\hp{{\h^\prime}}
\newcommand\hpp{\h^{\prime\prime}}
\newcommand\La{\Lambda}
\newcommand\dd[1]{\frac{\partial}{\partial #1}}
\newcommand\Der{\textup{Der}}
\newcommand\Span{\textup{Span}}
\newcommand\Id{\textup{Id}}
\newcommand\ad{{ad}}
\newcommand\End{\textup{End\,}}
\newcommand\Ker{\textup{Ker\,}}
\newcommand\Hom{\textup{Hom}}
\newcommand\gl{\mathfrak{gl}}
\newcommand\pd{\partial}
\newcommand\rpd{{}^*\partial}
\newcommand\p[1]{{\overline{#1}}}
\newcommand\D{\Delta}
\newcommand\AVZ{{ (\A \# \V)_0}}
\newcommand\eps{\epsilon}
\newcommand\hM{\widehat{M}}
\newcommand\hT{\widehat{T}}
\newcommand\DD{{\mathcal D}}
\newcommand\DL{\DD_\Lambda} 
\newcommand\supp{\textup{supp}} 
\newcommand\Ind{\textup{Ind}} 
\newcommand\Wz{W(m,n) \ltimes \A d_0}
\newcommand\sd{{\lambda_0}}

\newcommand\numberthis{\addtocounter{equation}{1}\tag{\theequation}}

\title
[Harish-Chandra $W(m,n)$-modules]
{Classification of simple strong Harish-Chandra $W(m,n)$-modules}
\author[Y.~Billig]{Yuly Billig}
\address{School of Mathematics and Statistics, Carleton University, Ottawa, Canada}
\email{billig@math.carleton.ca}
\author[v.~Futorny]{Vyacheslav Futorny}
\address{ Instituto de Matem\'atica e Estat\'\i stica,
Universidade de S\~ao Paulo,  S\~ao Paulo,
 Brasil}
 \email{vfutorny@gmail.com}
 \author[K.~Iohara]{Kenji Iohara}
\address{Universit\'e Lyon, Universit\'e Claude Bernard Lyon 1, CNRS UMR 5208, Institut Camille Jordan, 43 Boulevard du 11 Novembre 1918, F-69622 Villeurbanne cedex, France}
 \email{iohara@math.univ-lyon1.fr}
\author[I.~Kashuba]{Iryna Kashuba}
\address{ Instituto de Matem\'atica e Estat\'\i stica,
Universidade de S\~ao Paulo,  S\~ao Paulo,
 Brasil}
 \email{ikashuba@gmail.com}

\let\thefootnote\relax

\begin{abstract}
We classify all simple strong Harish-Chandra modules 
for the Lie superalgebra $W(m,n)$. We show that every such module is either strongly cuspidal 
or a module of the highest weight type. We construct tensor modules for $W(m,n)$, which are 
parametrized by simple finite-dimensional $gl(m,n)$-modules and show that every simple strongly cuspidal
$W(m,n)$-module is a quotient of a tensor module. Finally, we realize modules of the highest weight type
as simple quotients of the generalized Verma modules induced from tensor modules for $W(m-1,n)$.
\end{abstract}

\maketitle

\section{Introduction}
Consider a commutative algebra of Laurent polynomials $\R_m = $ \\
$ \C[t_1^{\pm 1}, \ldots, t_m^{\pm 1}]$ and a supercommutative Grassmann (exterior) algebra $\Lambda = \La(\xi_1, \ldots, \xi_n)$ with odd generators $\xi_1, \ldots, \xi_n$. We set $\A$ to be the commutative superalgebra $\R_m \otimes \La$. Our main object of study in this paper is the Lie superalgebra $\V = W(m,n)$ of derivations of $\A$.
We consider a category of \emph{Harish-Chandra modules}  over $W(m,n)$, which are \emph{weight} modules 
with finite weight multiplicities 
 and its subcategory of 
\emph{strong Harish-Chandra modules} with finite multiplicities of weight subspaces with respect to the part of the Cartan subalgebra corresponding to even variables. The  
strong Harish-Chandra modules are exactly the Harish-Chandra modules in the sense of \cite{MZ}, where the case of the superconformal Lie superalgebra $W(1,n)$ was studied for $n\geq 2$.
Previously, the case of   the Lie algebra of vector fields on a torus $W(m)$ was settled in \cite{BF}.

Crucial information about supports of simple strong Harish-Chandra modules for $W(m,n)$ may be obtained from 
the work of Volodymyr Mazorchuk and Kaiming Zhao \cite{MZh}. Even though \cite{MZh} only treats 
Lie algebra $W(m)$, the results of that paper apply in super setting as well.

We focus on an important class of \emph{cuspidal} strong Harish-Chandra modules, whose weight multiplicities are uniformly bounded. 
The key objects in this class are \emph{tensor} modules which generalize  tensor density modules, also known as the
intermediate series modules, for the Lie algebra of vector fields on a torus \cite{Ma}, \cite{MP2},  \cite{BF}.

Tensor modules belong to the subcategory of $\AV$-modules, they are parametrized by finite-dimensional simple $gl(m,n)$-modules and their supports $\lambda + \Z^m$, $\lambda \in \C^m$. For any $gl(m,n)$-module $V$, the tensor module $T(V, \lambda)=\A \otimes V $ comes with the following action of $\V = W(m,n)$:
\begin{align*}
(t^s f d_j) t^r\otimes g v =& (r_j + \lambda_j) t^{r+s} f g \otimes  v \\
&+  \sum_{i=1}^m s_i t^{r+s} f \otimes e_{ij} g v
 + \sum_{\alpha = 1}^n t^{r+s} (f)\rpd_\alpha \otimes e_{\alpha j} g v , \\
(t^s f \pd_\beta)  t^r \otimes g v =&  t^{r+s} f \frac{\pd g}{\pd \xi_\beta} \otimes v  \\
&+ \sum_{i=1}^m s_i  t^{r+s} f \otimes e_{i\beta} g v
+ \sum_{\alpha = 1}^n  t^{r+s} (f)\rpd_\alpha \otimes e_{\alpha \beta} g v.
\end{align*}   

Tensor modules $T(V, \lambda)$
  exhaust all simple weight modules of finite rank in the category of $\AV$-modules and provide a classification of simple cuspidal strong Harish-Chandra modules over $W(m,n)$.
  
  We also consider a subalgebra $\Wz \subset W(m+1, n)$ and define a tensor module for this subalgebra
as a tensor module for $W(m, n)$ on which $f d_0$ acts as multiplication by $\sd f$, $f \in \A$.

  We formulate our main result

\begin{theorem}
\label{tens-Introd}
Let $m, n$ be non-negative integers. Every simple strong Harish-Chandra $W(m+1, n)$-module is either
\begin{itemize}
\item[(1)] a quotient of a tensor module $T(V, \lambda)$, corresponding to a simple finite-dimensional 
$gl(m+1, n)$-module $V$ and $\lambda \in \C^{m+1}$, or
\item[(2)] a module of a highest weight type $L(T)^\theta$, 
twisted by an automorphism $\theta \in GL_{m+1} (\Z)$,
where $T$ is a simple quotient of a tensor module for $\Wz$.
\end{itemize}
\end{theorem}

% We will show the dichotomy between cuspidal and highest weight type modules  for arbitrary $m$ and $n$, and 
% hence complete the classification of all simple strong Harish-Chandra $W(m,n)$-modules,  in the forthcoming 
% paper.

The paper has the following structure: we discuss Harish-Chandra modules in Section 2 and their supports in Section 3. In Section 4 we consider the category of weight $\AV$-modules of a finite rank and exhibit the structure
of tensor modules - simple objects in this category. We study simple strongly cuspidal $W(m,n)$-modules in Section 5 and simple modules of the highest weight type in Section 6.

After this paper was completed, we learned about a recent preprint by Yaohui Xu and Rencai L\"u \cite{XL}, where the same results were obtained.

\section*{Acknowledgments}
Y.\,B.\ is supported in part by a Discovery grant from NSERC. Parts of this work was done during the visits of Y.\,B. to the University of S\~ao Paulo and to the University of Lyon. Y.\,B.  would like to thank these universities for their hospitality and FAPESP for funding the visit to USP.
V.\,F.\ is supported in part by the CNPq (304467/2017-0) and by the FAPESP (2018/23690-6).

\

\section{Definitions and Notations}

 Lie superalgebra $W(m,n)$ may be written as a free $\A$-module in the following way:
 \[ \bigoplus_{i=1}^m \A d_i \oplus \bigoplus_{\alpha=1}^n \A \pd_\alpha,
 \]
%$$\V = \sum\limits_{i=1}^m \A d_i \oplus \sum\limits_{\alpha=1}^n \A \pd_\alpha,$$
where $d_i = t_i \frac{\pd}{\pd t_i}$ and $\pd_\alpha = \frac{\pd}{\pd \xi_\alpha}$. We will use roman letters to index variables $t_1, \ldots, t_m$ and greek letters to index $\xi_1, \ldots, \xi_n$. 
% Set $w_\alpha = \xi_\alpha \pd_\alpha$.
We will be using multi-index notation, $t^r = t_1^{r_1} \times \ldots t_m^{r_m}$, for $r \in \Z^m$. 
We also set $|r| = |r_1| + \ldots + |r_m|$.

In addition to the usual derivation $\pd_\alpha$ of $\La$, satisfying
$\pd_\alpha (\xi_\alpha \xi^p) = \xi^p$ when $p_\alpha = 0$, we will also use
a right derivation $\rpd_\alpha$, satisfying $(\xi^p \xi_\alpha) \rpd_\alpha =
 \xi^p$ when $p_\alpha = 0$. We will write the symbol $\rpd_\alpha$ on the right,
in order to satisfy the usual sign conventions of superalgebras. This right derivation
has the following properties:
$$ (fg) \rpd_\alpha = f (g) \rpd_\alpha + (-1)^{\p{g}} (f) \rpd_\alpha g,$$
$$  (f) \rpd_\alpha = (-1)^{\p{f} + 1} \pd_\alpha (f). $$

Throughout the paper we will denote by $\p{f}$ the parity of $f$.

We will consider a chain of Lie superalgebras 
$W(m, n) \subset W(m,n) \ltimes \A d_0 \subset W(m+1,n)=\mathrm{Der}(\C[t_0^{\pm 1}] \otimes \A)$ where $d_0=t_0\dfrac{\partial}{\partial t_0}$. We will be using symbol $\V$ to denote one of these three algebras. 
It will be convenient to state some of our results in the context of $W(m,n)$, while others in the setting of $W(m+1, n)$.
Lie superalgebra $\V$ has a Cartan subalgebra $\h = \hp \oplus \hpp$, where
$\hp = \Span \{ d_1, \ldots, d_m \}$ for $\V=W(m,n)$ and $\hp=\Span \{ d_0, d_1, \ldots, d_m \}$ for $\V=W(m,n) \ltimes \A d_0$ or $\V = W(m+1, n)$, $\hpp = \Span \{ \xi_1 \pd_1, \ldots, \xi_n \pd_n \}$. 

% We will denote by $I$ the set $\{1, \ldots, m\}$ when $\V = W(m,n)$ and $\{0, \ldots, m\}$ when
% $\V=W(m,n) \ltimes \A d_0$ or $\V = W(m+1, n)$. With a slight abuse of notations, we will denote the 
% cardinality of this set, $m$ or $m+1$ by the same symbol $I$.

\begin{definition}
A $\V$-module $M$ is called a Harish-Chandra module if it has a weight decomposition with respect to $\h$ with finite-dimensional weight spaces.
\end{definition}

\begin{definition}
A Harish-Chandra module $M$ is called strong if it has finite-dimensional weight spaces with respect to $\hp$.
\end{definition}

For example, the adjoint module for $\V$ is a strong Harish-Chandra module. Note that strong Harish-Chandra modules coincide with Harish-Chandra modules in the sense of \cite{MZ}.

One immediately gets the following result.

\begin{lemma}\label{lem-HC-odd}
Suppose $M$ is a simple strong Harish-Chandra  $\V$-module. 
 Then the action of $\hpp$ on $M$ is diagonalizable and $M$ is a  Harish-Chandra module.
\end{lemma}

\begin{proof}
Since $M$ has a weight space decomposition with respect to $\hp$ with finite-dimensional weight spaces, 
the statement follows from the simplicity of $M$. 
\end{proof}

\

\section{Supports of strong Harish-Chandra modules}

 Cuspidal modules form an important subcategory of the Harish-Chandra modules.
 
 \begin{definition}
 A $\V$-module $M$ is called cuspidal (resp. strongly cuspidal) if $M$ is a Harish-Chandra (resp. strong Harish-Chandra) module 
 whose dimensions of weight spaces relative to $\h$ (resp. $\hp$) are uniformly bounded.
 \end{definition}
 
 We define {\it support} $\supp M$ of a strong Harish-Chandra $\V$-module $M$ as a set of weights 
 $\lambda \in \h^{\prime *}$ such that $M_\lambda \neq 0$. 
 
 Let us choose a basis $\{ \eps_i \}$ in $\h^{\prime *}$, dual to the basis $\{ d_i \}$ of $\hp$,  
 with weights $\eps_i$ defined by
 $\eps_i (d_j) = \delta_{ij}$.  
 We embed in $\h^{\prime *}$ integral lattice $\Z^{\dim \hp}$ spanned by $\{ \eps_i \}$.
 
 If a $\V$-module $M$ is indecomposable then its support lies in a single coset $\lambda + \Z^{\dim \hp}$,
 $\lambda \in \h^\prime$.
 
 Let us discuss twisting of a module by an automorphism of a Lie algebra. If $M$ is a $\V$-module and 
 $\theta$ is an automorphism of $\V$ then we can construct another $\V$-module $M^\theta$ with a new
 action of $\V$ on the same space $x.m = \theta(x) m$, for $x \in \V$, $m \in M$.
 
 The algebra of Laurent polynomials $\R_{m+1} = \C [t_0^{\pm 1}, t_1^{\pm 1}, \ldots, t_m^{\pm 1}]$ is 
 isomorphic to the group algebra $\C[\Z^{m+1}]$. The action of $GL_{m+1} (\Z)$ on $\Z^{m+1}$ extends to
 the action on $\R_{m+1}$ by $\theta (t^s) = t^{\theta(s)}$, $\theta \in GL_{m+1}(\Z)$, $s \in \Z^{m+1}$.

 This action also induces the action of $GL_{m+1} (\Z)$ on $\Der (\R_{m+1} ) = W(m+1, n)$
 by $\theta(\eta) (f) = \theta( \eta( \theta^{-1} (f)))$ for $\eta \in W(m+1, n)$, $f \in \R_{m+1}$.
 The action of $\theta$ on $\h^\prime$ induces action on $\h^{\prime *}$, which is given by the transpose of the inverse matrix ${\theta^{-1}}^T$ in basis $\{ \eps_i \}$. 
 If $M$ is a $W(m+1, n)$-module and $\theta \in GL_{m+1}(\Z)$ then $\supp M^\theta = {\theta^{-1}}^T (\supp(M))$.
 
Mazorchuk-Zhao \cite{MZ} described supports of simple Harish-Chandra $W(m, 0)$-modules. Their proofs literally apply to strong Harish-Chandra $W(m, n)$-modules yielding the following result:
\begin{theorem}
\label{thm_Mazorchuk-Zhao} 
(\cite{MZ}, Theorem 1)
Let $M$ be a simple strong Harish-Chandra $W(m+1, n)$-module. Then either

\noindent
(1) $M$ is a strongly cuspidal module, or

\noindent
(2) there exists $\theta \in GL_{m+1}(\Z)$ and $\lambda \in \supp M^\theta$ such that
$$\supp (M^\theta) \subset \lambda + \Z_+ \eps_0 + \Z \eps_1 + \ldots + \Z \eps_m$$
and the subspace
$$\mathop\oplus\limits_{\mu \in \lambda + \Z \eps_1 + \ldots + \Z \eps_m} (M^\theta)_\mu$$
is a simple cuspidal $W(m,n) \ltimes \R_m d_0$-module.
\end{theorem}

We will call a $W(m+1,n)$-module $M$ a module of a {\it highest weight type} if it satisfies condition (2) of the above theorem.

In the subsequent sections we will first study simple strongly cuspidal modules and then use their structure to get a
description of simple modules of the highest weight type. 

\

\section{$\AV$-modules of a finite type}

We begin in a general setting. Let $\A$ be an arbitrary commutative unital superalgebra and 
$\V$ be a Lie superalgebra. 
We assume that $\V$ is an $\A$-module and $\V$ acts on $\A$ by derivations, and the following
relation holds:
$$ [f \eta, g \tau] = f \eta(g) \tau - (-1)^{(\p{f} + \p{\eta})(\p{g} + \p{\tau})} g \tau(f) \eta
+ (-1)^{\p{\eta} \p{g}} fg [\eta, \tau] .$$ 
A prototypical example of this setting is when $\V$ is the algebra of derivations of $\A$.
The foundations of the theory we discuss in this section were laid by Rinehart \cite{R}.

\begin{definition}
An $\AV$-module $M$ is a vector space with actions of a unital commutative superalgebra $\A$ and  a Lie superalgebra $\V$, which are compatible via the Leibniz rule:
$$\eta (f m) = \eta(f) m + (-1)^{\p{\eta}\p{f}} f (\eta m),  \ \ 
f \in \A, \, \eta\in\V, \, m \in M.$$
\end{definition}

Equivalently, $\AV$-module structure may be expressed via the smash product $\A \# U(\V)$. View $U(\V)$ as a Hopf algebra with a coproduct $\Delta$, $\Delta(u) = \sum u_{(1)} \otimes u_{(2)}$. The smash product $\A \# U(\V)$ is the associative algebra structure on vector space $\A \otimes U(\V)$ with the product
$$ (f \otimes u) (g \otimes v) = \sum (-1)^{\p{u}_{(2)} \p{g} }
f u_{(1)} (g) \otimes u_{(2)} v .$$ 
% Notice that, for an $\AV$-module $M$, the $\A \# U(\V)$-module structure on $M$ is defined by 
% $$ (f\# u) m=f (u m), \ \ f \in \A, u \in U(\V), m \in M.$$

Then $\AV$-module structure is equivalent to the action of the associative algebra $\A \# U(\V)$.

\begin{definition}
The {\it algebra of differential operators} $\DD(\A, \V)$ is the quotient of $\A \# U(\V)$
by the ideal generated by the elements $f \# \eta - 1 \# f \eta$, $f \in \A$, $\eta \in \V$.
\end{definition}

 The following Lemma follows immediately from the definition:

\begin{lemma}
The subspace $\A \# \V = \A \otimes \V \subset \A \# U(\V)$ is a Lie subsuperalgebra with the Lie bracket
$$[f \otimes \eta, g \otimes \tau] = f \eta(g) \otimes \tau 
- (-1)^{(\p{f} + \p{\eta})(\p{g} + \p{\tau})} g \tau(f) \otimes \eta
+ (-1)^{\p{\eta} \p{g}} fg \otimes [\eta, \tau] .$$
\end{lemma}

% Let us now return to our settings where $\A = \R_m \otimes \Lambda$ and $\V = W(m,n)$ 
% {\color{blue}or $W(m,n)\ltimes \A d_0$}. 

\begin{lemma} 
Let $\A = \R_m \otimes \Lambda$ and $\V = W(m,n)$. Let $M$ be an $\AV$-module with a weight decomposition relative to $\hp$. Then $M$ is finitely generated over $\A$ if and only if it is a strong Harish-Chandra module. 
\end{lemma}
\begin{proof} 
Without loss of generality we may assume that $M$ is indecomposable. Since $\A$ is finitely generated over $\R_m$, being finitely generated over $\A$ is equivalent to being finitely generated over $\R_m$. Consider weight decomposition of $M$ relative to $\hp$:
$$M = \mathop\oplus\limits_{\mu \in \h^{\prime *}} M_\mu. $$
Note that $t^s \xi^r M_\mu \subset M_{\mu + s}$, 
 $t^s \xi^r d_i M_\mu \subset M_{\mu + s}$,
 $t^s \xi^r \pd_\alpha M_\mu \subset M_{\mu + s}$. 
% {\color{blue} for $0\; \text{or}\; 1\leq i\leq m$ and $1\leq \alpha\leq n$. 
% In particular, as $d_0$ is central, it acts as a scalar operator, say $c_M\mathrm{id}_M$, on $M$}.
We view elements of $\Z^m$ as linear functionals on $\hp$ via $s (d_i) = s_i$ for $1\leq i\leq m$.
Since $M$ is indecomposable, its support consists of a single coset $\lambda + \Z^m$.
Suppose that $M$ is finitely generated as an $\R_m$-module. We may assume that all generators are weight vectors. Since monomials $t^s \in \R_m$ act bijectively and transitively on the set of weight spaces, we may further assume that all generators belong to the same weight space $U = M_\lambda$. Then $M =\A \otimes_{\La} U \cong  \R_m \otimes U$ and we see that being finitely generated over $\R_m$ is equivalent to being a strong Harish-Chandra module.
\end{proof}
For the rest of this section we will assume that $\V = W(m,n)$ or $\V = \Wz$. we will be proving many results 
simultaneously for both of these algebras. When considering the case  $\V = W(m,n)$, all statements 
containing expressions with a subscript $0$, e.g., $d_0$, should be ignored. Throughout the section index $i$ 
will be assumed to run from 1 to $m$, and index $\alpha$ from 1 to $n$.

Let us assume that $M$ is a simple strong Harish-Chandra $\AV$-module. From the proof of the previous Lemma we see that $M \cong \A \otimes_{\La} U$, where $U = M_\lambda$
with $\lambda \in \h^{\prime *}$, $\dim U < \infty$.  
Without loss of generality we may assume that $\lambda \neq 0$.
Note that $U$ is a module over $\La$ with a weight decomposition with respect to $\hpp$.

 Lie superalgebra $\A \# \V$ has a Cartan subalgebra $1 \otimes \h$. The action of 
$1 \otimes \hp$ induces a $\Z^m$-grading on $\A \# \V$, let us denote by 
$(\A \# \V)_0$ the homogeneous component of degree zero, which is a Lie subsuperalgebra.

Lie superalgebra $(\A \# \V)_0$ is spanned by the following elements: \break
$f D_i (g, r) = t^{-r} f \otimes t^r g d_i$, \    
$f \Delta_\alpha (g, r) = t^{-r} f \otimes t^r g \pd_\alpha$, and
$f D_0 (g, r) = t^{-r} f \otimes t^r g d_0$, 
where $f, g \in \Lambda$, $r \in \Z^m$.

In the following Proposition, we record Lie brackets for the generators of  $(\A \# \V)_0$ 
as $\Lambda$-module. The proof is a straightforward calculation, and we omit it.

\begin{proposition}
\label{rel}
Elements $D_i (f, r)$, $\Delta_\alpha (f,r)$, $D_0 (f,r) \in (\A \# \V)_0$, satisfy the following commutator relations:
\begin{equation}
\label{DD}
\begin{split}
[D_i (f,r), D_j (g,s)] =
&s_i D_j (fg , r+s) - s_i f D_j (g, s) \\
&-r_j D_i (fg, r+s) + (-1)^{\p{f} \p{g}} r_j g D_i (f, r),
\end{split}
\end{equation}
\begin{equation}
\begin{split}
\label{DV}
[D_i (f,r), \D_\alpha (g,s)] =
&s_i \D_\alpha (fg , r+s) - s_i f \D_\alpha (g, s) \\
&- (-1)^{\p{f} + \p{g}} D_i (\pd_\alpha(f) g, r + s),
\end{split}
\end{equation}
\begin{align}
\label{VV}
[\D_\alpha (f,r), \D_\beta (g,s)] =
\D_\beta (f \pd_\alpha (g) , r+s)  
- (-1)^{\p{f} + 1 } \D_\alpha (\pd_\beta(f) g, r + s),
\end{align}
\begin{equation}
\label{DDz}
[D_i (f,r), D_0 (g,s)] = s_i D_0 (fg , r+s) - s_i f D_0 (g, s),
\end{equation}
\begin{equation}
\label{VDz}
[\D_\alpha (f,r), D_0 (g,s)] = D_0 (f \pd_\alpha(g), r + s),
\end{equation}
\begin{equation}
\label{DzDz}
[D_0 (f,r), D_0 (g,s)] = 0. 
\end{equation}
\end{proposition}

We have actions of a commutative superalgebra $\La$ and a Lie 
superalgebra $\AVZ$ on a finite-dimensional space $U = M_\lambda$, which satisfy the following compatibility properties:
\begin{align}
\label{LLambda1}
\eta (f u) =&\eta(f) u + (-1)^{\p{\eta}\p{f}} f (\eta u), \\
\label{LLambda2}
f (\eta u) =&(f \eta) u,
\end{align}
for $\eta \in \AVZ$, $f\in\La$, $u\in U$.
This means that $U$ admits the action of the algebra of differential operators $\DD(\La, \AVZ)$.

\begin{proposition}
\label{recover}
The action of $\A \# U(\V)$ on $M$ may be recovered from the joint actions of $\La$, $\AVZ$ on $U$:
\begin{align*}
%\label{recover1}
t^r f d_i (t^s \otimes u) =& t^r d_i (t^s) \otimes f u
+ t^{r+s} \otimes D_i (f,r) u \\
%\label{recover2}
t^r f \pd_\alpha (t^s \otimes u) =&  t^{r+s} \otimes \Delta_\alpha (f,r) u, \\
t^r f d_0 (t^s \otimes u) =& t^{r+s} \otimes D_0 (f,r) u \\
\end{align*} 
\end{proposition}
\begin{proof}
This follows immediately from the definition of $\AVZ$ and from $\AV$-module structure on $M$.
\end{proof}
We get an obvious corollary:
\begin{corollary}
$M = \A \otimes_\La U$ is an irreducible $\AV$-module if and only if $U$ is an irreducible $\DD(\Lambda, \AVZ)$-module.
\end{corollary}
\begin{remark}
% Let $\A = \R_m \otimes \Lambda$ and $\V = W(m, n)$  or $W(m,n) \ltimes \A d_0$.
% There is a natural action of $(\A \# \V)_0$ on $\Lambda$, as well as the action 
% of $\Lambda$ on $(\A \# \V)_0$.
% This allows us to consider the algebra of differential operators $\DD( \Lambda, (\A \# \V)_0 )$. 
It is not difficult to show that we have in fact an isomorphism of associative algebras:
\begin{equation}
\label{iso}
\A \# U(\V) \cong \A \mathop\otimes\limits_{\Lambda} \DD( \Lambda, (\A \# \V)_0 ).
\end{equation}
Here the algebra on the right is the quotient of $\A \# U(\AVZ)$ by the ideal generated by 
$a f \# x - a \# f x$, with $a \in \A$, $x \in \AVZ$, $f \in \Lambda$.
Isomorphism (\ref{iso}) allows us to interpret $M = \A \otimes_\La U$ as an $\A \# U(\V)$-module induced
from $\DD( \Lambda, (\A \# \V)_0 )$-module $U$. 
\end{remark}

Since $U$ is a weight space for $\V$, we also have
$$D_i (1, 0)\vert_U = \lambda_i \Id_U,  D_0 (1, 0)\vert_U = \lambda_0 \Id_U.$$

% The $\AV$-action  on $M = \A \otimes_{\La} U$ may be recovered from the joint actions of $\La$, $\AVZ$ on $U$ {\color{blue}where the $\V$-action is given by}:
% \begin{align}
% \label{recover1}
% t^r f d_i (t^s \otimes u) =& t^r d_i (t^s) \otimes f u
% + t^{r+s} \otimes D_i (f,r) u \\
% \label{recover2}
% t^r f \pd_\alpha (t^s \otimes u) =&  t^{r+s} \otimes \Delta_\alpha (f,r) u.
% \end{align} 
% Hence, $M$ is an irreducible $\AV$-module if and only if $U$ is an irreducible
% $(\Lambda, \AVZ)$-module.    {\color{green}Checked Up to here !!!}

% Denote by $\eps_1,  \ldots, \eps_m$ the standard basis of $\Z^m$.  {\color{red} To be continued !}

\begin{theorem}
Let $U$ be a finite-dimensional $\AVZ$-module. Then the action of $\AVZ$ is polynomial, i.e.,
\begin{align}
\label{E1}
D_i (f,r) =& \sum_{k\in \Z^m_+} \frac{r^k}{k!} d_i (f, k-\eps_i), \\
\label{E2}
\Delta_\alpha (f,r) =& \sum_{k\in \Z^m_+} \frac{r^k}{k!} \pd_\alpha (f, k),\\
\label{E3}
D_0 (f,r) =& \sum_{k\in \Z^m_+} \frac{r^k}{k!} d_0 (f, k),
\end{align}
where $d_i (f, k), \pd_\alpha (f, k), d_0 (f, k) \in \End U$ with $d_i (f, k) = 0, 
\pd_\alpha (f, k) = 0$, $d_0 (f, k) = 0$ for $k \gg 0$.
\end{theorem}

We use a shift by $\eps_i$ in the definition of $d_i (f, k)$ in order to exhibit a natural $\Z^m$-grading on a new Lie 
superalgebra that will be defined using the elements in the right hand sides of (\ref{E1})-(\ref{E3}).

\begin{proof}
By Theorem 3.1 in \cite{B2}, for $i = 1, \ldots, m$, the action of $\{ D_i (1, s) \}$ on $U$ is polynomial in $s$. Using Proposition \ref{rel}, we compute:
$$[D_i (1,s) , D_i (f, 0)] = -s_i D_i (f,s) + s_i f D_i (1,s).$$
Since the left hand side and the last summand of the right hand side are polynomial in $s$, we conclude that $s_i D_i (f,s)$ is  polynomial in $s$.

Next, 
\begin{align*}
[D_i (1,s-\eps_i) , D_i (f, \eps_i)] = & 2  D_i (f,s) -  D_i (f, \eps_i) \\
&-s_i D_i (f,s) + (s_i -1) f D_i (1,s-\eps_i).
\end{align*}
In the above equality, all terms except $D_i (f,s)$ are known to be polynomial, hence $D_i (f,s)$ is a polynomial in $s$.

From 
$$[D_i (1,s-\eps_i) , \Delta_\alpha (f, \eps_i)] = \Delta_\alpha  (f,s) -  
\Delta_\alpha  (f, \eps_i),$$
we conclude that $ \Delta_\alpha  (f,s)$ is also a polynomial in $s$.

In case when $\V = \Wz$, we also need to establish polynomiality of $D_0(f, s)$. This follows from the equality
$$[\D_\alpha(f, s), D_0 (\xi_a, 0)] = D_0 (f, s).$$
\end{proof}

\begin{theorem}
\label{brackets}
Operators $d_i (f, k)$, $\pd_\alpha(f, k)$, and $d_0 (f, k)$
acting on $U$ satisfy the following commutator relations:
\begin{equation}
\label{R1}
[ d_i (f, -\eps_i), d_j (g, -\eps_j) ] = 0,
\end{equation}
\begin{align*}
[ d_i (f, -\eps_i), d_j (g, k) ] =& (k_i + \delta_{ij}) d_j (fg, k-\eps_i) 
\numberthis \label{R2} \\
&-  (k_i + \delta_{ij}) f d_j(g, k-\eps_i),
\, {\text \ for \ } k \neq -\eps_j,
\end{align*}
\begin{align*}
[ d_i (f, \ell), d_j (g, k) ] =& k_i d_j (fg, \ell + k)
\numberthis \label{R3} \\
& - \ell_j d_i(fg, \ell + k),
\, {\text \ for \ } \ell \neq -\eps_i, k \neq -\eps_j, 
\end{align*}
\begin{align*}
\label{R4} \numberthis
[ d_i (f, -\eps_i), \pd_\alpha (g, k) ] =& k_i \pd_\alpha (fg, k-\eps_i) 
 -  k_i f \pd_\alpha (g, k-\eps_i) \\
& - (-1)^{\p{f}+\p{g}} d_i(\pd_\alpha (f) g, k -\eps_i),
\end{align*}
\begin{align*}
\label{R5}  \numberthis
[ d_i (f, \ell), \pd_\alpha (g, k) ] =& k_i \pd_\alpha (fg, \ell + k) \\
&- (-1)^{\p{f}+\p{g}} d_i (\pd_\alpha (f) g, \ell + k),
\, {\text \ for \ } \ell \neq -\eps_i,
\end{align*}
\begin{align*}
\label{R6}  \numberthis
[ \pd_\alpha (f, \ell), \pd_\beta (g, k) ] =& \pd_\beta  (f\pd_\alpha(g), \ell + k) \\
&- (-1)^{\p{f}+1} \pd_\alpha (\pd_\beta (f) g, \ell + k),
\end{align*}
\begin{align}
\label{R0a}  
[ d_i (f, -\eps_i), d_0 (g, k) ] = k_i  d_0 (fg, k-\eps_i) -  k_i f d_0 (g, k-\eps_i) 
\end{align}
\begin{align}
\label{R0b}
[ d_i (f, \ell), d_0 (g, k) ] = k_i d_0 (fg, \ell + k),
\, {\text \ for \ } \ell \neq -\eps_i
\end{align}
\begin{align}
\label{R0c}  
[ \pd_\alpha (f, \ell), d_0 (g, k) ] =d_0  (f\pd_\alpha(g), \ell + k),
\end{align}
\begin{align}
\label{R0d}  
[ d_0 (f, \ell), d_0 (g, k) ] = 0,
\end{align}
\begin{equation}
\label{R7}
[ \pd_\alpha (f, 0), g ] = f\pd_\alpha(g),  
\end{equation}
 \begin{equation}
\label{R8}
[ \pd_\alpha (f, k), g ] = 0,  
\, {\text \ for \ } k \neq 0,
\end{equation}
 \begin{equation}
\label{R9}
[ d_i (f, k), g ] = 0, \ [ d_0 (f, k), g ] = 0.
\end{equation}
\end{theorem}

\begin{proof}
Let us derive relations (\ref{R1})-(\ref{R3}). We take polynomial expansions of (\ref{DD}):
\begin{align*}
&\left[ \sum_{p \in \Z^m_+} \frac{s^p}{p!} d_i (f, p -\eps_i),
\sum_{q \in \Z^m_+} \frac{r^q}{q!} d_j (g, q -\eps_j) \right] \\
&{\hskip 1cm}= r_i \sum_{l \in \Z^m_+} \frac{(r+s)^l}{l!} d_j (fg, l - \eps_j) 
-  r_i \sum_{l \in \Z^m_+} \frac{r^l}{l!} f d_j (g, l - \eps_j) \\
&{\hskip 1cm}- s_j \sum_{l \in \Z^m_+} \frac{(r+s)^l}{l!} d_i (fg, l - \eps_j) 
+ (-1)^{\p{f} \p{g}} s_j \sum_{l \in \Z^m_+} \frac{s^l}{l!} g d_i (f, l - \eps_i) . 
\end{align*}
Two polynomials in $2m$ variables that have equal values on $\Z^{2m}_+$ must be equal as 
polynomials, which allows us to equate their coefficients. Equating the terms with 
$s^0 r^0$ we immediately get (\ref{R1}). Extracting terms with $s^0 r^q$ with $q\neq 0$ we get
\begin{align*}
\frac{1}{q!} [ d_i(f, -\eps_i), &d_j(g, q - \eps_j) ] \\
&= \frac{1}{(q-\eps_i)!} d_j (fg, q - \eps_i - \eps_j) 
-   \frac{1}{(q-\eps_i)!} f d_j (g, q - \eps_i - \eps_j). 
\end{align*}
Multiplying both sides by $q!$ and setting $k = q - \eps_j$, we obtain (\ref{R2}).
In the same way, equating the terms with $s^p r^q$ where $p, q \neq 0$, we establish 
(\ref{R3}). 
Derivation of the remaining relations of Theorem \ref{brackets} is analogous, and we omit these calculations.
\end{proof}

Denote by $\LL$ Lie superalgebra with basis 
$$\{ \xi^r d_i (\xi^s, k-\eps_i), \, \xi^r d_0 (\xi^s, k), \, \xi^r \pd_\alpha (\xi^s, k) \}$$   
where $r, s \in \{ 0, 1 \}^n$, $k\in\Z^m_+$, $1 \leq i \leq m$, $1\leq \alpha \leq n \}$
and Lie brackets defined by (\ref{R1})-(\ref{R9}).
In fact, $\LL$ may be defined as a jet Lie superalgebra corresponding to $\AVZ$ (see 
\cite{B3} for the general construction of the jet Lie algebra).

Note that elements $\{ d_i (1, -\eps_i), d_0 (1,0),  \}$ are central in $\LL$.

Lie algebra $\LL$ has a Cartan subalgebra 
$$\HH = \Span \{ d_i (1,0), \pd_\alpha(\xi_\alpha, 0), d_0 (1,0) \}$$
which is diagonalizable in the adjoint representation:
\begin{align*}
[ d_i (1,0), \xi^r d_j (1, -\eps_j) ] =& 0, \\
[ \pd_\alpha (\xi_\alpha,0), \xi^r d_j (1, -\eps_j) ] =& 
r_\alpha \xi^r d_j (1, -\eps_j), \\
[ d_i (1,0), \xi^s d_j (\xi^r, -\eps_j) - \xi^s \xi^r d_j (1, -\eps_j)  ] =&
- \delta_{ij} \left( \xi^s d_j (\xi^r, -\eps_j) - \xi^s \xi^r d_j (1, -\eps_j) \right), \\
[ \pd_\alpha (\xi_\alpha,0),  \xi^s d_j (\xi^r, -\eps_j) - \xi^s \xi^r d_j (1, -\eps_j)  ] =&
(s_\alpha + r_\alpha)   \left( \xi^s d_j (\xi^r, -\eps_j) - \xi^s \xi^r d_j (1, -\eps_j) \right), \\
[ d_i (1,0), \xi^r d_j (\xi^s, k) ] =& k_i  \xi^r d_j (\xi^s, k)
{\text \ for \ } k \neq -\eps_j, \\
[ \pd_\alpha (\xi_\alpha,0),  \xi^r d_j (\xi^s, k) ] =& (r_\alpha + s_\alpha)
  \xi^r d_j (\xi^s, k),
{\text \ for \ } k \neq -\eps_j, \\
[ d_i (1,0), \xi^r \pd_\beta (\xi^s, k) ] =& k_i   \xi^r \pd_\beta (\xi^s, k), \\
[ \pd_\alpha (\xi_\alpha,0),  \xi^r \pd_\beta (\xi^s, k) ] =& 
(r_\alpha + s_\alpha - \delta_{\alpha \beta})  \xi^r \pd_\beta (\xi^s, k), \\
[ d_i (1,0), \xi^r d_0 (\xi^s, k) ] =& k_i  \xi^r d_0 (\xi^s, k), \\
[ \pd_\alpha (\xi_\alpha,0),  \xi^r d_0 (\xi^s, k) ] =& (r_\alpha + s_\alpha)  \xi^r d_0 (\xi^s, k),
\end{align*}
and $d_0 (1, 0)$ is central.

{\bf Example.} Consider $\AV$-module $M = W(m, n)$ with the adjoint action of $W(m, n)$,
natural left multiplication action of $\A$, and in case when $\V = \Wz$, we define the action of
$a d_0$ as left multiplication by $\sd a$, $a \in \A$, where $\sd$ is some fixed scalar, $\sd \in \C$.  
Let us fix $r \in\Z^m$ and take as $U$ the corresponding root space:
$$ U = \sum_{i=1}^m t^r \La d_i \oplus \sum_{\alpha=1}^n t^r \La \pd_\alpha.$$
We see that $U$ is a free $\La$-module of rank $m+n$. Let us compute the action
of $\AVZ$ on $U$:
\begin{align*}
D_i (f,s) t^r g d_j =& t^{-s} [t^s f d_i, t^r g d_j] = r_i t^r fg d_j - s_j t^t fg d_i,
\\
D_i (f,s) t^r g \pd_\beta =&  r_i t^r fg \pd_\beta
 - (-1)^{\p{f}+\p{g}} t^r \pd_\beta(f) g d_i,
\\
\Delta_\alpha (f,s) t^r g d_j =& t^r f \pd_\alpha (g) d_j
 - (-1)^{\p{g}} s_j t^r fg \pd_\alpha,
\\
\Delta_\alpha (f,s) t^r g \pd_\beta =&  t^r f \pd_\alpha (g)  \pd_\beta
+ (-1)^{\p{f}} t^r \pd_\beta(f) g \pd_\alpha ,
\\
D_0 (f,s) t^r g d_j =& \sd t^r fg d_j,
\\
D_0 (f,s) t^r g \pd_\beta =&  \sd t^r fg \pd_\beta.
\end{align*}

Taking formal expansions (\ref{E1})-(\ref{E3}) of $D_i (f,s)$, $\Delta_\alpha (f,s)$, and $D_0 (f,s)$
 in powers of $s$ and comparing them with the right hand sides of the above equalities, we derive the action on $U$ of Lie superalgebra $\LL$:
\begin{align*}
d_i (f, -\eps_i) t^r g d_j =& r_i t^r fg d_j ,  \\
d_i (f, -\eps_i) t^r g \pd_\alpha =&   r_i t^r fg \pd_\alpha 
- (-1)^{\p{f} + \p{g}} t^r \pd_\beta (f) g d_i, \\
d_i (f, \eps_a - \eps_i) t^r g d_j =& - \delta_{aj} t^r f g d_i, \\
d_i (f, \eps_a - \eps_i) t^r g \pd_\alpha =& 0, \\
d_i (f, k-\eps_i) =& 0 \text{ \ for \ } |k| > 1, \\
\pd_\alpha (f, 0) t^r g d_j =& t^r f \pd_\alpha(g) d_j, \\
\pd_\alpha (f, 0) t^r g \pd_\beta =&  t^r f \pd_\alpha(g) \pd_\beta 
+ (-1)^{\p{f}} t^r \pd(f) g \pd_\alpha , \\
\pd_\alpha (f, \eps_a) t^r g d_j =& - \delta_{aj} (-1)^{\p{g}} t^r f g \pd_\alpha ,\\
\pd_\alpha (f, \eps_a) t^r g \pd_\beta =& 0 , \\
\pd_\alpha(f, k ) =& 0 \text{\ for \ } |k| > 1, \\
d_0 (f, 0) t^r g d_j =& \sd  t^r fg d_j, \\
d_0 (f, 0) t^r g  \pd_\alpha =& \sd  t^r fg  \pd_\alpha, \\
d_0 (f, k) =& 0 \text{\ for \ } |k| > 0.
\end{align*}

Denote by $J_{r, \sd}$ the kernel in $\LL$ of this representation. To make this ideal independent of $r$ and $\sd$, we set $J = \mathop\cap\limits_{r\in\Z^m, \sd\in\C} J_{r, \sd}$.
Next Proposition is an immediate consequence of the above formulas:
\begin{proposition}
\label{rel}
(i) The quotient Lie superalgebra $\LL/J$ is a free $\La$-module with generators
$$ \{ d_i (1,-\eps_i), d_i (\xi_\beta, -\eps_i), d_i (1, \eps_j - \eps_i),
\pd_\alpha (1, 0), \pd_\alpha(\xi_\beta, 0), \pd_\alpha(1, \eps_j), d_0(1, 0) \}$$
with 
$i, j = 1, \ldots, m, \alpha, \beta = 1, \ldots, n.$

(ii) The following relations determine $\LL/J$ as a $\La$-module:
\begin{align*}
&d_i (f, k-\eps_i) = 0, \ \pd_\alpha (f, k) = 0 \text{ \ for \ } |k| > 1,  \\
&d_i (f, \eps_j - \eps_i) = f d_i (1, \eps_j - \eps_i), \\
&\pd_\alpha (f, \eps_j) = f \pd_\alpha (1, \eps_j), \\
&d_i (f, -\eps_i) = f d_i (1, -\eps_i) + \sum_{\beta=1}^n 
(f) \rpd_\beta \left( d_i (\xi_\beta, -\eps_i) -  \xi_\beta d_i (1, -\eps_i) \right) , \\
&\pd_\alpha (f, 0) = f \pd_\alpha (1, 0) + \sum_{\beta=1}^n 
(f) \rpd_\beta \left( \pd_\alpha (\xi_\beta, 0) -  \xi_\beta \pd_\alpha (1, 0) \right), \\
&d_0 (f, k) = 0 \text{ \ for \ } |k| > 0, \\
& d_0 (f, 0) = f d_0 (1, 0). 
\end{align*}
\end{proposition}

\begin{theorem}
\label{kernel}
Ideal $J$ is in the kernel of any finite-dimensional simple $\LL$-module.
\end{theorem}

\begin{proof}
Let $U$ be a  finite-dimensional simple $\LL$-module.
Consider the following elements of the Cartan subalgebra of $\LL$:
$$I^\prime = \sum_{i=1}^m d_i (1, 0), \ 
I^{\prime\prime} = \sum_{\alpha=1}^m \pd_\alpha (\xi_\alpha, 0),
\ I = I^\prime + I^{\prime\prime} .$$
Each of these elements defines a $\Z$-grading on $\LL$, and each of these gradings starts at $-1$.

 Applying Lemma 3.4 from \cite{B2}, to $\Z$-grading induced by $I^\prime$, we see that for some $N$ a proper ideal containing
$$\{ f d_i (g, k), \ f \pd_\alpha (g, k), \  f d_0 (g, k) \, | \, |k| > N \}$$
vanishes in $U$. The quotient by this ideal is a finite-dimensional Lie algebra, which acts on $U$, and we can now apply the following result:
\begin{lemma} (\cite{CK}, Lemma 1). 
Let ${\mathfrak g}$ be a finite-dimensional Lie superalgebra and let ${\mathfrak n}$ be a solvable ideal of ${\mathfrak g}$. Let ${\mathfrak a}$ be an even subalgebra of ${\mathfrak g}$ such that ${\mathfrak n}$ is a completely reducible $ad ({\mathfrak a})$-module with no trivial summand. Then ${\mathfrak n}$ acts trivially in any irreducible finite-dimensional ${\mathfrak g}$-module $U$.
\end{lemma}
From Proposition \ref{rel} we can see that $\ad(I)$ acts on $J$ in a diagonal way with positive integer eigenvalues. This implies that 
 the image of $J$ in the finite-dimensional quotient of $\LL$ is nilpotent
and by previous Lemma, $J$ vanishes in every simple finite-dimensional $\LL$-module module $U$.
\end{proof}

The following Proposition may be shown by a direct computation using Proposition \ref{rel} and Theorem \ref{brackets}.

\begin{proposition}
As a $\La$-module, $\LL / J$ is generated by three supercommuting subalgebras:

(1) Central ideal spanned by $\{ d_i (1, -\eps_i), d_0 (1, 0) \}$,

(2) Abelian subalgebra spanned by  $\{ \pd_\alpha (1, 0) \}$,

(3) A subalgebra isomorphic to $gl(m,n)$ and spanned by the following 
elements:
\begin{align*}
&e_{ij} = d_j (1, \eps_i - \eps_j), \\
&e_{\alpha\beta} = \pd_\beta (\xi_\alpha, 0) - \xi_\alpha \pd_\beta (1, 0), \\
&e_{i\beta} = \pd_\beta (1, \eps_i), \\
&e_{\alpha j} = d_j (\xi_\alpha, -\eps_j) - \xi_\alpha d_j (1, -\eps_j).
\end{align*}
\end{proposition}

\begin{theorem}
Let $U$ be a simple finite-dimensional module for $\DD(\La, \LL)$.
% vector space which admits the action of a commutative unital superalgebra $\La$ and a Lie superalgebra $\LL$
% where $\La$-module structures on $U$ and $\LL$ are compatible
% \begin{equation*}
% (f \eta) u = f (\eta u), \text{\ for \ } \eta \in\LL, f\in \La, u\in U,
% \end{equation*}
% and that satisfy (\ref{R7})-(\ref{R9}). Suppose that $U$ has no non-trivial subspaces, invariant under both 
% actions. 
Assume that operators 
$\{ d_i (1,-\eps_i) \}$ act on $U$ as scalars $\lambda_i \Id_U$ with $\lambda \neq 0$,
and $d_0(1, 0)$ acts on $U$ as $\sd \Id_U$.

Then there exists a finite-dimensional irreducible $gl(m,n)$-module $(V, \rho)$ 
such that $U \cong \La \otimes V$. The action of $\LL$ integrates to the action 
of $\AVZ$ on $\La \otimes V$ in the following way:
\begin{align*}
&D_j (f,s) = \, \, \lambda_j f \otimes \Id_V  \, + \,
\sum_{i=1}^m s_i f \otimes \rho (e_{ij}) 
 + \sum_{\alpha = 1}^n (f)\rpd_\alpha \otimes \rho(e_{\alpha j}) , \\
&\Delta_\beta (f,s) =  f \frac{\pd}{\pd \xi_\beta}  \otimes \Id_V + \sum_{i=1}^m s_i f \otimes \rho (e_{i\beta}) 
+ \sum_{\alpha = 1}^n (f)\rpd_\alpha \otimes \rho(e_{\alpha \beta}), \\
&D_0 (f,s) = \, \, \sd f \otimes \Id_V.
\end{align*}  
\end{theorem}
\begin{proof}
Since elements $f  d_i (1,-\eps_i) \in \LL$ act on $U$ as multiplication by $\lambda_i f$ and $\lambda_i \neq 0$ for some $i$, we see that $\La$-action is encoded in the action of $\LL$. Hence $U$ is an irreducible finite-dimensional $\LL$-module. By Theorem  \ref{kernel} it is a simple $\LL/J$-module.

We are going to show that subalgebras spanned by $\{ f d_i (1, -\eps_i) \}$ and
 $\{ \pd_\alpha (1, 0) \}$ define a structure of a free $\La$-module on $U$, with 
$ f d_i (1, -\eps_i)$ acting as multiplications by $\lambda_i f$ and 
$\pd_\alpha (1, 0)$ acting as $\frac{\pd}{\pd\xi_\alpha}$.

Let $V$ be the joint kernel of  $\{ \pd_\alpha (1, 0) \}$. It is easy to see that this subspace is non-zero. Indeed, we can start with any non-zero vector in $U$ and act repeatedly by elements $\{ \pd_\alpha (1, 0) \}$. Since these operators anticommute and $\pd_\alpha (1, 0)^2 = 0$, we will end up with a non-zero vector in $V$.
Since $\{ \pd_\alpha (1, 0) \}$ commute with the $gl(m,n)$ subalgebra in $\LL$, 
$V$ admits a structure of a $gl(m,n)$-module. 

Set $V_p = \xi^p V$. By (\ref{R7}) we have 
$$ [ \pd_\alpha (1, 0), \xi^p ] = \pd_\alpha (\xi^p) .$$
The standard differentiation argument shows that such subspaces form a direct sum:
$$ \mathop\oplus\limits_{p \in \{ 0,1\}^n} V_p \cong \La \otimes V. $$
This direct sum is invariant under $\LL/J$, hence it coincides with $U$.

If $V^\prime$ is a $gl(m,n)$-submodule in $V$ then $\La \otimes V^\prime$ is an
$\LL/J$-submodule in $U$. Hence $V$ must be irreducible as a $gl(m,n)$-module.
It is also easy to see that the converse holds as well: if $V$ is irreducible as a $gl(m,n)$-module then $U = \La \otimes V$ is irreducible as an $\LL/J$-module.
Indeed, starting with an arbitrary non-zero vector in $U$ we can apply a sequence 
of operators $\{ \pd_\alpha (1, 0) \}$ to get a non-zero vector in $V$. Then using 
$gl(m,n)$ action we can get all of $V$, and finally acting with elements $\{ f d_i (1, -\eps_i) \}$ we can generate $U$. This shows that any non-zero vector in $U$ generates $U$, hence $U$ is an irreducible $\LL/J$-module.

Using expansions (\ref{E1}),  (\ref{E2}) and combining our results about the action of $\LL$ on $U$ we obtain the last claim of the Theorem.
\end{proof}

% \begin{definition}
% We call $U$ a $(\Lambda, \LL)$-module, if it admits actions of Lie superalgebra $\LL$, commutative 
% superalgebra $\Lambda$, which satisfy (\ref{LLambda1})-(\ref{LLambda2}) and the action
% of $\LL$ on $\Lambda$ is given by  (\ref{R7})-(\ref{R9}).
% \end{definition}

Now we can obtain a description of simple weight $\AV$-modules of a finite rank, recovering the action of $\V$ from the action of $\AVZ$ by applying Proposition \ref{recover}.

\begin{theorem}
\label{tens}
Let $\V = W(m,n)$, $\A = \R_m \otimes \Lambda$.
 Every simple weight $\AV$-module of a finite rank is a tensor module. Such modules are parametrized by finite-dimensional simple $gl(m,n)$-modules $V$ and their supports $\lambda + \Z^m$. A tensor module is a tensor product
$$ \A \otimes V $$
with the following action of $\V = W(m,n)$:
\begin{align*}
(t^s f d_j) t^r g v =& (r_j + \lambda_j) t^{r+s} f g v \\
&+  \sum_{i=1}^m s_i t^{r+s} f  \rho (e_{ij}) g v
 + \sum_{\alpha = 1}^n t^{r+s} (f)\rpd_\alpha  \rho(e_{\alpha j}) g v , \\
(t^s f \pd_\beta)  t^r g v =&  t^{r+s} f \frac{\pd g}{\pd \xi_\beta}  v  \\
&+ \sum_{i=1}^m s_i  t^{r+s} f \rho (e_{i\beta}) g v
+ \sum_{\alpha = 1}^n  t^{r+s} (f)\rpd_\alpha \rho(e_{\alpha \beta}) g v.
\end{align*}   
\end{theorem}

We also get an analogous statement for $\Wz$.

\begin{theorem}
\label{tenz}
Let $\V = \Wz$, $\A = \R_m \otimes \Lambda$.
Every simple weight $\AV$-module $M$ of a finite rank is a tensor module for $W(m, n)$ as described in Theorem
\ref{tens}, $M = \A \otimes V$, with $\A d_0$ acting by multiplication on the left tensor factor:
$$(t^s f d_0) t^r g v = \sd t^{r+s} fg v $$
for some scalar $\sd \in \C$. 
\end{theorem}

We will refer to modules described in Theorem \ref{tenz} as tensor $\AV$-modules for $\V = \Wz$.  

\begin{theorem}
\label{cat}
Let $\V = W(m, n)$ or $\V = \Wz$ and $\A = \R_m \otimes \Lambda$.
A category of weight $\AV$-modules of a finite rank supported on $\lambda + \Z^m$ 
and $d_0$ acting as multiplication by $\lambda_0$
is equivalent to the category of finite-dimensional $\DD(\La, \LL)$-modules with the central elements $\{ d_i (1, -\eps_i) \}$ acting as multiplications by $\lambda_i$ and $\{ d_0 (1, 0) \}$ acting as multiplications by $\lambda_0$.

Given such a $\DD(\La, \LL)$-module $U$, the corresponding $\AV$-module is $M = \R_m \otimes U$ with the action of $\V = W(m,n)$ given as follows:
\begin{align*}
(t^s f d_j) t^r u =& r_j t^{r+s} f u 
+  \sum_{k \in\Z_+^m } \frac{s^k}{k!}  t^{r+s} \rho (d_j (f, k-\eps_j) ) u , \\
(t^s f \pd_\beta)  t^r u =&  \sum_{k \in\Z_+^m } \frac{s^k}{k!}  t^{r+s} 
\rho (\pd_\beta (f, k) )u  ,\\
(t^s f d_0) t^r u =& \sd t^{r+s} f u.
\end{align*}   
For every finite-dimensional $\LL$-module $U$ the sums in the right hand sides of the above formulas are finite. 
\end{theorem}

\section{Cuspidal $W(m,n)$-modules}

The goal of this section is to prove the following:

\begin{theorem}
\label{quot}
Every non-trivial simple cuspidal strong Harish-Chandra module for $\V = W(m,n)$ or $\V = \Wz$ is a quotient of a tensor $\AV$-module.
\end{theorem}

In order to prove this result, we construct a functor from the category of $\V$-modules to the category of $\AV$-modules,
preserving the property of being (strongly) cuspidal. We will be following the ideas developed in \cite{BF}.

We begin with the coinduction functor. For a $\V$-module $M$, the coinduced $\AV$-module is the space $\Hom (\A, M)$
with the $\AV$-action defined in Lemma \ref{coind} below. Note that $\Hom (\A, M)$ is a $\Z_2$-graded space with
$\Hom_0 (\A, M)$ consisting of parity-preserving maps, and $\Hom_1 (\A, M)$ consisting of parity-reversing maps.

\begin{lemma}
\label{coind} Let $M$ be a $\V$-module. Then 

(a) $\Hom (\A, M)$ is an $\AV$-module with the action defined as follows:
\begin{align*}
(f \varphi) (g) =& (-1)^{\p{f} \p{\varphi}} \varphi (fg), \\
(\eta \varphi) (g) =& \eta \varphi(g) - (-1)^{\p{\eta}\p{\varphi}} \varphi(\eta(g)),
\end{align*}
where $\varphi \in \Hom (\A, M)$, $\eta \in \V$, $f,g\in \A$.

(b) There exists a $\V$-module homomorphism $\pi: \ \Hom (\A, M) \rightarrow M$,
defined by $\pi(\varphi) = \varphi(1)$.
\end{lemma}

The proof given in Proposition 4.3 in \cite{BF}, generalizes to super setting in a straightforward way.

The coinduced module is too big for our purposes, it is not even a weight module. It does have a maximal weight submodule
$$ \mathop\bigoplus_{\beta \in \h^*} \Hom (\A, M)_\beta,$$
where
$$ \Hom (\A, M)_\beta = \left\{ \varphi \in \Hom (\A, M) \, | \, \forall \alpha \in \h^* \ \varphi(\A_\alpha) \subset M_{\alpha+\beta} \right\}.$$

Still, this weight submodule is also too big -- it is not cuspidal when $M$ is. To remedy this situation we define the {\it $\AV$-cover}
$\hM$ of a $\V$-module as a subspace in $\Hom (\A, M)$:
$$\hM = \Span \{ \psi(\tau, m) \, | \, \tau\in\V, m\in M \},$$
where $\psi(\tau, m)$ is defined as
$$\psi(\tau, m) g = (-1)^{(\p{\tau} + \p{m})\p{g}} (g \tau) m .$$
Again, a generalization of the argument of Proposition 4.5 in \cite{BF} to super setting, shows that $\hM$ is a weight $\AV$-module. 

The map $\hM \rightarrow M$, $\psi(\tau, m) \mapsto \tau m$, is a homomorphism of $\V$-modules with the image $\V M$.

\begin{theorem}
\label{cusp}
If $M$ is a cuspidal (resp. strongly cuspidal) $\V$-module then $\hM$ is a cuspidal (resp. strongly cuspidal) $\AV$-module.
\end{theorem}

The proof of this theorem is based on the following 
\begin{proposition}
\label{ann}
For any cuspidal $\V$-module $M$ there exists $N \in \N$ such that for all $p, q \in \Z^m$, $r \in \{0,1\}^n$,
$i, j = 1,\ldots, m$, $\alpha = 1, \ldots, n$, the following elements of $U(\V)$ annihilate $M$:
\begin{align}
\label{difA}
&\sum_{a=0}^N (-1)^a {N \choose a} (t^{p} t_i^a \xi^r d_j) (t_i^{q-a} d_i), \\
\label{difB}
&\sum_{a=0}^N (-1)^a {N \choose a} (t^{p} t_i^a \xi^r \pd_\alpha) (t_i^{q-a} d_i), \\
\label{difC}
&\sum_{a=0}^N (-1)^a {N \choose a} (t^{p} t_i^a \xi^r d_0) (t_i^{q-a} d_i).
\end{align}
\end{proposition}

\begin{proof}
We begin by establishing (\ref{difA}) with $i=j$. Consider a subalgebra in $\V$ spanned by elements $t_i^k d_i$, $k \in \Z$. This subalgebra is isomorphic to Witt algebra $W(1)$. With respect to the action of this subalgebra, $M$ decomposes into a direct sum of cuspidal $W(1)$-modules whose dimensions of weight spaces (with respect to $d_i$)  have the same bound over all direct summands. By Corollary 3.4 in \cite{BF}, there exists $\ell \in \N$ such that for all $p, q \in \Z$ the following elements of $U(W(1))$ 
annihilate $M$:
$$ \Omega^{(\ell)}_{p,q}  = \sum_{a=0}^\ell  (-1)^a {\ell \choose a} (t_i^{p+a} d_i) (t_i^{q-a} d_i).$$ 
Let $s \in \Z^m$ and $r \in \{0,1\}^n$. Then 
$\Gamma(s, p, q) = [t^s \xi^r d_i, \Omega^{(\ell)}_{p,q}]$ also annihilates $M$.
We have
\begin{align*}
\Gamma(s, p, q) = &\sum_{a=0}^\ell (-1)^a {\ell \choose a}
(p+a-s_i) (t^s t_i^{p+a} \xi^r d_i)  (t_i^{q-a} d_i) \\
+ &\sum_{a=0}^\ell (-1)^a {\ell \choose a}
(q-a-s_i)  (t_i^{p+a} d_i)  (t^s t_i^{q-a} \xi^r d_i) .
\end{align*}
Consider expression 
$$\Phi(s, p, q) = \Gamma(s + 2\eps_i, p, q-2) - 2  \Gamma(s + \eps_i, p, q-1) 
+ \Gamma(s, p, q),$$
which also annihilates $M$.
In this expression the second parts of $\Gamma$'s will cancel, yielding
\begin{align*}
\Phi(s, p, q) = &\sum_{a=2}^{\ell+2} (-1)^a {\ell \choose a-2}
(p + a - s_i - 4)  (t^s t_i^{p+a} \xi^r d_i)  (t_i^{q-a} d_i) \\
- 2 &\sum_{a=1}^{\ell+1} (-1)^{a-1} {\ell \choose a-1}
(p + a - s_i - 2)  (t^s t_i^{p+a} \xi^r d_i)  (t_i^{q-a} d_i) \\
+ &\sum_{a=0}^{\ell} (-1)^a {\ell \choose a}
(p + a - s_i)  (t^s t_i^{p+a} \xi^r d_i)  (t_i^{q-a} d_i) .
\end{align*}
Finally, consider $\Phi(s - \eps_i, p+1, q) - \Phi(s, p ,q)$, which gives us the following 
expression annihilating $M$:
\begin{align*}
&2 \sum_{a=0}^{\ell+2} (-1)^a 
\left(  {\ell \choose a-2} + 2  {\ell \choose a-1} +  {\ell \choose a} \right)
(t^s t_i^{p+a} \xi^r d_i)  (t_i^{q-a} d_i) \\
=&2 \sum_{a=0}^{\ell+2} (-1)^a  {\ell + 2 \choose a} 
(t^s t_i^{p+a} \xi^r d_i)  (t_i^{q-a} d_i) .
\end{align*}
This establishes (\ref{difA}) with $i=j$ and $N =\ell + 2$. The cases of (\ref{difA}) with $i \neq j$, 
(\ref{difB}) and (\ref{difC}) are actually slightly easier and follow exactly the same lines,
where we consider commutators $[t^s \xi^r d_j, \Omega^{(\ell)}_{p,q}]$,
 $[t^s \xi^r \pd_\alpha, \Omega^{(\ell)}_{p,q}]$, and  $[t^s \xi^r d_0, \Omega^{(\ell)}_{p,q}]$
as the starting points. 
% We consider 
% $\Gamma^\prime (s, p, q) = - [t^s \xi^r d_j, \Omega^{(\ell)}_{p,q}$. Then
% \begin{align*}
% \Gamma^\prime (s, p, q) =&
% s_i \sum_{a=0}^\ell (-1)^a {\ell \choose a}
% (t^s t_i^{p+a} \xi^r d_j)  (t_i^{q-a} d_i) \\
% + s_i &\sum_{a=0}^\ell (-1)^a {\ell \choose a}
% (t_i^{p+a} d_i)  (t^s t_i^{q-a} \xi^r d_j) .
% \end{align*}
% \end{proof}
% Forming expression
% $$\Phi^\prime (s, p, q) = \Gamma^\prime (s + 2\eps_i, p, q-2) - 2  \Gamma^\prime % (s + \eps_i, p, q-1) 
% + \Gamma^\prime (s, p, q),$$
% we eliminate the second sums:
% \Phi^\prime (s, p, q) = &\sum_{a=2}^{\ell+2} (-1)^a {\ell \choose a-2}
% (s_i - 2)  (t^s t_i^{p+a} \xi^r d_j)  (t_i^{q-a} d_i) \\
% - 2 &\sum_{a=1}^{\ell+1} (-1)^{a-1} {\ell \choose a-1}
% (s_i - 1)  (t^s t_i^{p+a} \xi^r d_j)  (t_i^{q-a} d_i) \\
% + &\sum_{a=0}^{\ell} (-1)^a {\ell \choose a}
% s_i  (t^s t_i^{p+a} \xi^r d_j)  (t_i^{q-a} d_i) .
% \end{align*}
\end{proof}

\begin{proof}
Now we can use Proposition \ref{ann} to give a proof of Theorem \ref{cusp}.
The proofs for the cuspidal and strongly cuspidal cases are virtually identical, so we will only consider the strongly cuspidal case.
Without loss of generality we may assume that $M$ is an indecomposable module. Then its support belongs to a single coset $\lambda + \Z^m \subset \hp^*$. We choose the coset representative $\lambda$ in such a way that for each $i = 1, \ldots, m$ either
$\lambda(d_i) = 0$ or $\lambda(d_i) \not\in\Z$. Clearly, the support of $\AV$-cover 
$\hM$ is the coset $\lambda + \Z^m$ and for $k \in \Z^m$ the weight space 
$\hM_{\lambda + k}$ is spanned by 
$\psi( t^{k-s} \xi^r d_j, M_{\lambda + s})$, $\psi( t^{k-s} \xi^r \pd_\alpha, M_{\lambda + s})$ 
and $\psi( t^{k-s} \xi^r d_0, M_{\lambda + s})$
with $s \in \Z^m$, $r \in \{ 0, 1\}^n$, $j = 1, \ldots, m$, $\alpha = 1, \ldots n$. We need to show that weight spaces of $\hM$ are in fact finite-dimensional and their dimensions are uniformly bounded.

Consider a norm on $\Z^m$, $\left\| s \right\| = \max\limits_i |s_i|$. We claim that 
$\hM_{\lambda + k}$ is spanned by
\begin{equation}
\label{span}
\left\{ \psi( t^{k-s} \xi^r d_j, M_{\lambda + s}) , \ \psi( t^{k-s} \xi^r \pd_\alpha, M_{\lambda + s}),
 \psi( t^{k-s} \xi^r d_0, M_{\lambda + s})  \ | \ \left\| s \right\|  \leq \frac{N}{2} \right\},
\end{equation}
where $N$ is the constant from Proposition \ref{ann}. 

Let us show by induction on $|s_1| + \ldots + |s_m|$  that elements \break
$\psi( t^{k-s} \xi^r d_j, M_{\lambda + s})$,
$\psi( t^{k-s} \xi^r \pd_\alpha, M_{\lambda + s})$,
and  $\psi( t^{k-s} \xi^r d_0, M_{\lambda + s})$
belong to the span of (\ref{span}).

If $\left\| s \right\|  \leq N/2$ then there is nothing to prove. Otherwise
there exists $i$ such that $|s_i| > N/2$. Let us assume $s_i > N/2$, the case $s_i <  -N/2$ is analogous. Let $u$ be an arbitrary vector from $M_{\lambda + s}$. It follows from our assumptions that $\lambda(d_i) + s_i \neq 0$. Hence there exists 
$v \in M_{\lambda + s}$ such that $d_i v = u$.

Let us show that 
\begin{align*}
&\psi( t^{k-s} \xi^r d_j, u) = - \sum_{a=1}^N (-1)^a {N \choose a}
\psi( t^{k-s} t_i^a \xi^r d_j, (t_i^{-a}d_i) v), \\
&\psi( t^{k-s} \xi^r \pd_\alpha, u) = - \sum_{a=1}^N (-1)^a {N \choose a}
\psi( t^{k-s} t_i^a \xi^r \pd_\alpha, (t_i^{-a}d_i) v), \\
&\psi( t^{k-s} \xi^r d_0, u) = - \sum_{a=1}^N (-1)^a {N \choose a}
\psi( t^{k-s} t_i^a \xi^r d_0, (t_i^{-a}d_i) v).
\end{align*}
Indeed, this follows immediately from the fact that operators from Proposition \ref{ann}
annihilate vector $v$. Since vectors $ (t_i^{-a}d_i) v$ have weights $\lambda + s - a \eps_i$ with $a = 1, \ldots, N$, by induction assumption the right hand sides in the above 
equalities belong to the span of (\ref{span}). This proves our claim, which now easily
provides a uniform bound on the dimensions of weight spaces of $\hM$.
\end{proof}

To prove Theorem \ref{quot} we consider the projection $\pi: \ \hM \rightarrow M$ given
by Lemma \ref{coind}(b), $\pi (\psi(\tau, u)) = \tau (u)$. Since $M$ is a simple module with a non-trivial action of  $\V$, this map is surjective.
By Theorem \ref{cusp}, the $\AV$-cover $\hM$ is strongly cuspidal, hence it has a Jordan-H\"older series 
$$ (0) = \hM_0 \subset \hM_1 \subset \hM_2 \subset \ldots \subset \hM_s = \hM$$
with the quotients that are strongly cuspidal simple $\AV$-modules.
Choosing the largest $k$ with $\hM_{k-1} \subset \Ker \pi$, $\hM_{k} \not\subset \Ker \pi$,
we obtain a surjective $\V$-module homomorphism from a simple strongly cuspidal $\AV$-module $\hM_k / \hM_{k-1}$ onto $M$.  By Theorem \ref{tens}, a simple strongly cuspidal $\AV$-module is a tensor module, which completes the proof of Theorem \ref{quot}.

\begin{remark}
We point out that a simple tensor $\AV$-module for $\V = \Wz$ with $\sd \neq 0$ remains simple as a $\V$-module, since in this case the action of $\A$ is encoded by $\A d_0$, hence any $\V$-submodule is invariant under the action of $\A$.
\end{remark}

\section{$W(m+1, n)$-modules of the highest weight type}

Consider a $\Z$-grading on $\V = W(m+1, n)$ by the eigenvalues of $\ad (d_0)$. The zero component
of this grading is $\V_0 = \Wz$. This grading induces a triangular decomposition 
$$ \V = \V_- \oplus \V_0 \oplus \V_+.$$

Let $T$ be a strongly cuspidal $\V_0$-module. We let $\V_+$ act on $T$ trivially and define a generalized Verma module $M(T)$ as the induced module
$$ M(T) = \Ind_{\V_0 \oplus \V_+}^{\V} T \cong U(\V_-) \otimes T.$$
The adjoint action of $\{d_0, d_1, \ldots, d_m\}$ induces a $\Z^{m+1}$-grading on $M(T)$ and every $\V$-submodule in $M(T)$ is homogeneous with respect to this grading.
It is clear that for $m > 0$ the generalized Verma module $M(T)$ is not Harish-Chandra. 

We define the radical $M^{rad}$ of $M(T)$ as a (unique) maximal $\V$-submodule intersecting with $T$ trivially,
and we set $L(T)$ to be the quotient 
$$L(T) = M(T) / M^{rad}.$$

\begin{theorem}
Let $\V = W(m+1, n)$ and let $T$ be a strongly cuspidal module for $\V_0 = \Wz$ such that 
$\V_0 T = T$. Then $L(T)$ is a strongly cuspidal $\V$-module.
\end{theorem}
\begin{proof}
Consider an $\A^\prime \V^\prime$-cover $\hT$ of $T$ with $\A^\prime = \R_m$, $\V^\prime = \Wz$. By Theorem \ref{cusp}, $\hT$ is a strongly cuspidal module for $\Wz$. It follows from Theorem \ref{cat} that 
$\hT$ is a polynomial $\Z^m$-graded module (see \cite{BB} or \cite{BZ} for the definition of a polynomial module). By Theorem 1.12 in \cite{BB}, $L(\hT)$ is a strong Harish-Chandra module (although the results of
\cite{BB} are proved in non-super setting, all statements and their proofs extend to the super case).
Since $T$ is a quotient of $\hT$, it follows that $L(T)$ is a quotient of $L(\hT)$, and hence $L(T)$ is a strong Harish-Chandra module.
\end{proof}

Now we combine all of the results of this paper to establish our main Theorem \ref{tens-Introd}.
Let $M$ be a simple strong Harish-Chandra $W(m+1, n)$-module. By Theorem \ref{thm_Mazorchuk-Zhao},
$M$ is either strongly cuspidal or of the highest weight type. If $M$ is strongly cuspidal, by Theorem \ref{quot}
it is a quotient of a tensor $W(m+1, n)$-module, whose structure is given by Theorem \ref{tens}.

If $M$ is a simple module of the highest weight type, it is isomorphic to module $L(T)$, twisted by an automorphism $\theta \in GL_{m+1} (\Z)$. Here $T$ is a simple strongly cuspidal module for $\Wz$ described in Theorem \ref{tenz}.

\end{document}